\documentclass[12pt]{amsart}

\usepackage{amssymb, amsfonts, amsthm} 
\usepackage{graphicx}
\usepackage{color}

\newtheorem{theorem}[equation]{Theorem}
\newtheorem{lemma}[equation]{Lemma}
\newtheorem{corollary}[equation]{Corollary}
\newtheorem{proposition}[equation]{Proposition}

\numberwithin{equation}{section}

\theoremstyle{definition}

\newtheorem{remark}[equation]{Remark}

\newtheorem{problem}[equation]{Problem}
\newtheorem{conjecture}[equation]{Conjecture}

\newcommand{\be}{\begin{equation}}
\newcommand{\ee}{\end{equation}}

\DeclareMathOperator{\diam}{diam}
\DeclareMathOperator{\dist}{dist}
\newcommand{\Rn}{{\mathbb R}^n}

\newcommand{\Bn}{{\mathbb B}^n}
\newcommand{\Rt}{{\mathbb R}^2}

\newcommand{\psubset}{\varsubsetneq}

\newcommand{\comment}[1]{}

\newcounter{minutes}
\divide\time by 60
\newcounter{hours}
\multiply\time by 60 \addtocounter{minutes}{-\time}

\begin{document}

\title[Subdomain geometry of hyperbolic type metrics]
{Subdomain geometry of hyperbolic type metrics}
\author[R. Kl\'en]{R. Kl\'en}
\address{Riku Kl\'en, Department of Mathematics and Statistics, University of Turku,
FIN-20014 Turku, Finland}
\email{riku.klen@utu.fi}
\author[Y. Li]{Y. Li}
\address{Yaxiang Li, Department of Mathematics,
Hunan Normal University, Changsha,  Hunan 410081, People's Republic
of China} \email{yaxiangli@163.com}
\author[M. Vuorinen]{M. Vuorinen}
\address{Matti Vuorinen, Department of Mathematics and Statistics, University of Turku,
FIN-20014 Turku, Finland}
\email{vuorinen@utu.fi}




\begin{abstract}
Given a domain $G \subsetneq \Rn$ we study the quasihyperbolic and the distance ratio metrics of $G$ and their connection to the corresponding metrics of a subdomain $D \subset G$. In each case, distances
in the subdomain are always larger than in the original domain.
Our goal is to show that, in several cases, one can prove
a stronger domain monotonicity statement.
We also show  that under special hypotheses we have
inequalities in the opposite direction.
\end{abstract}

\maketitle

\def\thefootnote{}
\footnotetext{ \texttt{\tiny File:~\jobname .tex,
          printed: \number\year-\number\month-\number\day,
          \thehours.\ifnum\theminutes<10{0}\fi\theminutes }
} \makeatletter\def\thefootnote{\@arabic\c@footnote}\makeatother




\section{Introduction}\label{sec-1}


Recently many authors have studied what we call ''hyperbolic type metrics'' of a domain $G \subset \Rn$ \cite{HIMPS,HPS,Klen,Lin,RasilaTalponen12,Va05}. Some of the examples are the Apollonian metric, the M\"obius invariant metric, the quasihyperbolic metric and the distance ratio metric.
The term ''hyperbolic type metric'' is for us just a descriptive
term, we do not define it. The term is justified by the fact that the metric is similar to the hyperbolic metric of the unit ball $\Bn\,.$
In this paper we will study a hyperbolic type metric
 $m_G$  with the following  two properties:
\begin{enumerate}
  \item if $D \subset G$ is a subdomain, then $m_D(x,y) \ge m_G(x,y)$ for all $x,y \in D$,\label{subdomain condition}
  \item sensitivity to the boundary variation: if $x_0 \in G$ and $D = G \setminus \{ x_0 \}$, then the metrics $m_G$ and $m_D$ are quite different close to $x_0$ whereas "far away" from $x_0$ we might expect that they are nearly equal (see Remark \ref{rem-intro-1} \eqref{rem-intro}).
\end{enumerate}
In particular, we require that $m_G$ is defined for every proper
subdomain of $\Rn\,.$
The purpose of this paper is to study the subdomain monotonicity property \eqref{subdomain condition} and to prove conditions under which we have a quantitative refinement of \eqref{subdomain condition}.

For a subdomain $G \psubset {\mathbb R}^n$ and
$x,y\in G$ the {\em distance ratio
metric $j_G$} is defined by
$$ j_G(x,y)=\log\left(1+\frac{|x-y|}{\min\{\delta_G(x),\delta_G(y)\}}\right)\,,
$$ where $\delta_G(x)$ denotes the Euclidean distance from $x$ to $\partial G$. Sometimes
we abbreviate $\delta_G$ by writing just $\delta\, .$
The above form of the $j_G$ metric, introduced in \cite{Vu2}, is obtained by a slight modification
of a metric that was studied in \cite{GO,GP}.
The {\em quasihyperbolic metric} of $G$ is defined by
the quasihyperbolic length minimizing property
$$k_G(x,y)=\inf_{\gamma\in \Gamma(x,y)} \ell_k(\gamma),
\quad \ell_k(\gamma) =\int_\gamma \frac{|dz|}{\delta_G(z)}\,,
$$
where $\Gamma(x,y)$ represents the family of all rectifiable paths
joining $x$ and $y$ in $G$, and $\ell_k(\gamma)$ is the quasihyperbolic length of
$\gamma$ (cf. \cite{GP}).
For a given pair of points $x,y\in G,$ the infimum is always
attained \cite{GO}, i.e., there always exists a quasihyperbolic
geodesic $J_G[x,y]$ which minimizes the above integral,
$k_G(x,y)=\ell_k(J_G[x,y])\,$ and furthermore with the property that
the distance is additive on the geodesic: $k_G(x,y)=$
$k_G(x,z)+k_G(z,y) $ for all $z\in J_G[x,y]$. If the domain $G$ is
emphasized we call $J_G[x,y]$ a $k_G$-geodesic. In this paper, our main work is to refine some inequalities between $k_G$ metric, $j_G$ metric and the Euclidean metric. Both the distance ratio and the quasihyperbolic metric qualify as hyperbolic type metrics because
 \begin{itemize}
 \item[$\bullet$] both are defined for every proper subdomain of $\Rn\,,$
 \item[$\bullet$] for the case of the unit ball $\Bn$ both are comparable to the hyperbolic metric of $\Bn\,,$ see Section \ref{sec-2} below,
  \item[$\bullet$] it is well-known that both metrics satisfy the above properties (1) and (2).
 \end{itemize}
 These metrics have recently been studied, e.g., in \cite{HIMPS,Klen,RasilaTalponen12}. We mainly study the following three problems and our main results will be given in Section \ref{sec-2}, Section \ref{sec-3} and Section \ref{sec-4} respectively.

\begin{problem}\label{prob1}
For some special domains, can we obtain certain upper estimates for the quasihyperbolic metric in terms of the distance ratio metric?
\end{problem}

Indeed, inequalities of this type were used to characterize so called
$\varphi$-domains in \cite{Vu2}.

\begin{problem}\label{prob2}
Is there some relation between $k$ metric and the Euclidean metric? The same question can be asked for $j$ metric and the Euclidean metric?
\end{problem}

Let $G_1$ and $G_2$ be proper subdomains of $\Rn$. It is well know that if $G_1\subset G_2$ then for all $x,y\in G_1$, $$j_{G_1}(x,y)\geq j_{G_2}(x,y)$$ and $$k_{G_1}(x,y)\geq k_{G_2}(x,y).$$ We expect some better results to hold for some special class of domains. This motivates the following question.

\begin{problem}\label{pro}
Let $G_1\subset G_2$ be two proper subdomains in $\Rn$ such that $\partial G_1\cap\partial G_2$ is either $\varnothing$ or a discrete set.  Does there  exist a constant $c>1$ such that for all $x,y\in G_1$, the following holds:
\begin{equation}\label{eq0}
  m_{G_1}(x,y)\geq c m_{G_2}(x,y),
\end{equation}
where $m_{G_i} \in \{ j_{G_i},k_{G_i} \}$ for $i=1,2.$
\end{problem}

Our main results for Problem \ref{prob1} are Theorems \ref{newlem} and \ref{newlemma2}, for Problem \ref{prob2}  Theorems \ref{Jung-appl} and \ref{Jung-appl-j} and for Problem \ref{pro}  Theorems \ref{thm1'} and \ref{thm2}. We also formulate some open problems and conjectures.
Finally, it should be pointed out that there are many more metrics
for which the above problems could be studied. For some of these metrics, see \cite{Vu07}.


\section{Results concerning Problem \ref{prob1}}\label{sec-2}

In this section, we study Problem \ref{prob1} and our mains results are Theorems \ref{newlem} and \ref{newlemma2}.
The following proposition, which will be used in the proof of Theorems \ref{newlem}, gathers together several
basic well-known properties of the metrics $k_G$ and $j_G$, see for instance \cite{GP,Vu}.
The motivation comes from the well-known inequality
\begin{equation}\label{j-le-k}k_G(x,y) \ge \log\left(1+ \frac{L}{\min\{\delta(x), \delta(y)\}}\right) \ge j_G(x,y)\,,
\end{equation}
for a domain $G\psubset {\mathbb R}^n, x,y \in G,$
where $L= \inf \{ \ell(\gamma) : \gamma\in \Gamma(x,y) \}\,.$
One can ask: when do both the metrics $j_G$ and $k_G$ (or $\rho_{B^n}$) coincide\,?
\begin{proposition}\label{k-equals-j}
\begin{enumerate}
\item\label{ksv-sub1} For $x\in B^n$, we have
$$k_{B^n}(0,x)=j_{B^n}(0,x)=\log \frac{1}{1-|x|}\,.
$$
\item\label{ksv-sub2} Moreover, for $b\in S^{n-1}$ and $0<r<s<1$, we have
$$k_{B^n}(br,bs)=j_{B^n}(br,bs)=\log \frac{1-r}{1-s}\,.
$$
\item\label{ksv-sub3} Let $G\psubset\Rn$ be a domain and $z_0\in G$.
Let $z\in\partial G$ be such that $\delta(z_0)=|z_0-z|$. Then for all
$u,v\in [z_0,z]$, we have

$$k_G(u,v)=j_G(u,v)= \left|\log
\frac{\delta(z_0)-|z_0-u|}{\delta(z_0)-|z_0-v|}\right|
=\left|\log \frac{\delta(u)}{\delta(v)}\right|.$$
\item\label{rhoj} For $x,y\in B^n$ we have
$$ j_{B^n}(x,y)\le \rho_{B^n}(x,y) \le 2 j_{B^n}(x,y)
$$
with equality on the right hand side when $x=-y\,.$
\end{enumerate}
\end{proposition}
\begin{proof}
(1) We see from (\ref{j-le-k}) that
$$ j_{B^n}(0,x) = \log\frac{1}{1-|x|} \le k_{B^n}(0,x)
\le \int_{[0,x]} \frac{|dz|}{\delta(z)}= \log\frac{1}{1-|x|}
$$
and hence $[0,x]$ is the $k_{B^n}$-geodesic
between $0$ and $x$.

The proof of (\ref{ksv-sub2}) follows from (\ref{ksv-sub1}) because
the quasihyperbolic length is additive along a geodesic
$$ k_{B^n}(0,bs) = k_{B^n}(0,br)+ k_{B^n}(br,bs)\,.
$$

The proof of (\ref{ksv-sub3}) follows from (\ref{ksv-sub2}).

The proof of (\ref{rhoj}) is given in \cite[Lemma 7.56]{AVV}.
\end{proof}

The hyperbolic geometry of $B^n$ serves as model for the quasihyperbolic geometry and we will use below a few basic
facts of the hyperbolic metric  $\rho_{B^n}$ of  $B^n\,.$
These facts appear in standard textbooks of hyperbolic geometry
and also in \cite[Section 2]{Vu}. 
For the case of $B^n$, we make use of an explicit formula
 \cite[(2.18)]{Vu} to the effect that for
$x,y\in B^n$
\begin{equation} \label{rhodef}
\sinh \frac{\rho_{B^n}(x,y)}2=\frac{|x-y|}{t}\, ,
t=\sqrt{(1-|x|^2)(1-|y|^2)}\,.
\end{equation}

It is readily seen that 
$$\rho_{B^n} \le  2 k_{B^n} \le 2 \rho_{B^n}$$
and it is well-known by \cite[Lemma 7.56]{AVV} that a similar inequality also holds for $j_{B^n}$
$$j_{B^n} \le   \rho_{B^n} \le 2 j_{B^n} \,.$$

\begin{remark}
The proofs of Proposition~\ref{k-equals-j}~(1) and (2) show that the diameter
$(-e,e), e \in S^{n-1},$ of $B^n$ is a geodesic of $k_{B^n}$ and hence the
quasihyperbolic distance is additive on a diameter. At the same time
we see that the $j$ metric is additive on a radius of the unit ball
but not on the full diameter because for $x \in B^n \setminus \{ 0 \}$
$$ j_{B^n}(-x,x) < j_{B^n}(-x,0) + j_{B^n}(0,x) \,.$$
\end{remark}

In order to obtain certain upper estimates for the quasihyperbolic metric, in terms
of the distance ratio metric, we present the
following theorem.

\begin{theorem}\label{newlem}
\begin{enumerate}
\item\label{jk}
For $0<s<1$ and $x,y\in B^n(s)$, we have
$$ j_{B^n}(x,y) \le k_{B^n}(x,y) \le (1+s)\, j_{B^n}(x,y).
$$
\item\label{ksv-sub4} Let $G\psubset\Rn$ be a domain, $w \in G\,,$ and
$w_0\in (\partial G) \cap S^{n-1}(w,\delta(w))\,. $ If $s  \in (0,1)$ and
$x,y \in B^n(w, s \delta(w))\cap [w,w_0]$,
then we have
$$  k_G(x,y)\le (1+s) j_G(x,y)\,.$$
\end{enumerate}
\end{theorem}

\begin{proof} (1) Fix  $x,y\in B^n(s)$ and the geodesic $\gamma$
of the hyperbolic metric joining them. Then  $\gamma \subset B^n(s)$ and for
all $w\in B^n(s)$ we have
$$ \frac{1}{1-|w|} < \frac{1+s}{2} \frac{2}{1-|w|^2} \,.
$$
Therefore, by Proposition \ref{k-equals-j} (4)
$$ k_{B^n}(x,y) \le \int_{\gamma} \frac{|dw|}{1-|w|} \le \frac{1+s}{2}
\int_{\gamma} \frac{2|dw|}{1-|w|^2}\le \frac{1+s}{2}\,
\rho_{B^n}(x,y) \le (1+s) j_{B^n}(x,y)\,.
$$
for $x,y\in B^n(s)$.
The inequality, $j_{B^n}(x,y) \le k_{B^n}(x,y)$, follows from (\ref{j-le-k}).

For the proof of (\ref{ksv-sub4}) set $B=B^n(w, \delta(w))\,.$ Then
by part (\ref{jk})
$$
k_G(x,y) \le k_B(x,y)\le (1+s) j_B(x,y) \le (1+s) j_G(x,y) \,.
$$

This completes the proof of the theorem.
\end{proof}

\begin{remark} Theorem \ref{newlem} (1) 
 refines the well-known inequality in \cite[Lemma 2.11]{Vu1} and \cite[Lemma 3.7(2)]{Vu} for the case of $B^n$. We have
been unable to prove a similar statement for a general domain.
However, a similar result for ${\mathbb R}^n \setminus \{0\}$
is obtained in the sequel (see Theorem~\ref{newlemma2}). To obtain this, we collect some basic properties.
\end{remark}

Martin and Osgood \cite{mo} proved the following explicit formula: for $x,y \in \Rn \setminus \{0\}$
\begin{equation}\label{kinpunctspace}
  k_{\Rn \setminus \{ 0 \}}(x,y) = \sqrt{\alpha^2 + \log^2 (|x|/|y|)},
\end{equation}
where $\alpha = \measuredangle (x,0,y)$.

We here introduce a lemma which is a modification of \cite[Lemma 4.9]{KlenB}.
\begin{lemma}\label{klenlemma}
Let  $z \in G = \Rn \setminus \{ 0 \}$ and $k_G(x,z) = k_G(y,z)$ with $|z| \le |x|,|y|$.
Then $\measuredangle (x,z,0) < \measuredangle (y,z,0)$ implies $|x-z| < |y-z|$.
\end{lemma}
\begin{proof}
 Let $k_G(x,z)=r$. By \eqref{kinpunctspace} the angle $\measuredangle (x,z,0)$ determines the point
$x$ uniquely up to a rotation about the line through 0 and $z$. By symmetry and similarity
it is sufficient to consider only the case $n = 2$ and $z = e_1$. We will show that the function
  $$
    f(s) = |x(s)-e_1|^2
  $$
  is strictly increasing on $(0,\min \{ r,\pi \})$, where
$$x(s) = (e^s \cos \phi(s),e^s \sin \phi(s))\quad\mbox{ with }
\varphi(s)=\sqrt{\min \{ r,\pi \}^2-s^2}\,.
$$
For $s \in [0,\min \{ r,\pi \}]$, a simple calculation gives
  $$
    f(s) = |x(s)|^2+1-2|x(s)| \cos \phi(s) = e^{2s}+1-2 e^s \cos \phi(s)
  $$
   and hence
  $$
    f'(s) = 2 e^s \left( e^s-\cos \phi(s)-\frac{s \sin \phi(s)}{\phi(s)} \right).
  $$

  If $s \in (0,\min \{ r,\pi \})$, then we see that
  $$
    e^s-\cos \phi(s)-\frac{s \sin \phi(s)}{\phi(s)} \ge e^s-\cos \phi(s)-s \ge e^s-1-s > 0
  $$
  and hence $f'(s) > 0$ is implying the assertion.
\end{proof}

\begin{theorem}\label{newlemma2}
  Let $G = \Rn \setminus \{0\}$. Then
  \begin{enumerate}
    \item for $\alpha \in (0,\pi]$ and $x,y \in G$ with $\measuredangle (x,0,y) \le \alpha$
    \[
      k_G(x,y) \le \frac{\alpha}{\log (1+2\sin(\alpha/2))}j_G(x,y) \le (1+\alpha)j_G(x,y).
    \]
    \item for $\varepsilon > 0$, $x \in G$ and $y \in B^n(|x|/t) \cup (\Rn \setminus \overline{B(t|x|)})$
    \[
      k_G(x,y) \le (1+\varepsilon)j_G(x,y),
    \]
    where $t=\exp((1+1/\varepsilon) \log 3)$.
  \end{enumerate}
\end{theorem}
\begin{proof}
  (1) We may assume that $|y| \ge |x|$. Fix $k_G(x,y) = c > 0$. Now $j_G(x,y) = \log (1+|x-y|/|x|)$
and by Lemma \ref{klenlemma} the quantity $k_G(x,y)/j_G(x,y)$ attains its maximum when $\alpha$ is maximal,
which is equivalent to $|y| = |x|$. Thus,
  \[
    \frac{k_G(x,y)}{j_G(x,y)} \le \frac{\alpha}{\log \left( 1+\frac{2 |x| \sin(\alpha/2)}{|x|} \right)}
= \frac{\alpha}{\log \left( 1+2 \sin \frac{\alpha}{2} \right)}
  \]
  and the first inequality follows.

  Let us next prove the second inequality. We define the functions $f$ and $g$ by
$$f(x) = \log (1+x)~~\mbox { and }~~ g(x) = x/(1+x/2)\,.
$$
Because
  \[
    g'(x) = \frac{4}{(2+x)^2} \le \frac{1}{1+x} = f'(x),
  \]
  $g'(x) > 0$ for $x \ge 0$ and $f(0) = 0 = g(0)$, we have $g(x) \le f(x)$. Thus
  \[
    \frac{\alpha}{\log (1+2 \sin (\alpha / 2))} \le \frac{\alpha}{\frac{2 \sin (\alpha / 2)}{1+ \sin (\alpha / 2)}}
 = \frac{\alpha}{2} \left( 1+\frac{1}{\sin (\alpha / 2)} \right).
  \]
  The function $h(\alpha) = (\alpha /(2) \left( 1+1/(\sin (\alpha / 2)) \right)$ is convex, since
  \[
    h''(\alpha) = \frac{\alpha (3+\cos \alpha)-4 \sin \alpha}{16 \sin^3 (\alpha / 2)} \ge 0.
  \]
  Therefore, $h(\alpha) \le \max \{ h(0),h(\pi) \} = \pi$ on $[0,\pi]$ and $h(\alpha) \le 1+(1-1/\pi)\alpha \le 1+ \alpha$
both imply the assertion.

  (2) We prove that
    \[
      k_G(x,-u x) \le (1+\varepsilon)j_G(x,-u x),
    \]
    where $u \in (0,1/t]$ or $u > t$. We may assume $x = e_1$. Now
  \[
    \left( \frac{k_G(x,y)}{j_G(x,y)} \right)^2 = \frac{\pi^2+\log^2 (1/u)}{\log^2 ((|x|+u|x|+u)/u)}
\ge \frac{\log^2 (1/u)}{\log^2 (3/u)} = A
  \]
  and $A\ge 1+\varepsilon$ is equivalent to $u \le 1/t$ or $u \ge t$. The assertion follows from \eqref{kinpunctspace}.
\end{proof}

\begin{remark} \label{rem-intro-1}
\begin{enumerate}
\item  In Theorem \ref{newlemma2} (1), the constant $h(\alpha) = \alpha / \log(1+2 \sin (\alpha/2))$ appears with the bound
$h(\alpha) \le 1+\alpha$. This upper bound of $h(\alpha)$ is not sharp as can be seen from the proof. By computer simulations,
 we obtained that the sharp upper bounds are $h(\alpha) \le 1+((1/\log 3)-(1/\pi)) \alpha$ for
$\alpha \in [0,\pi]$ and $h(\alpha) \le 1+\pi\alpha/(2 \log(1+\sqrt{2}))$ for $\alpha \in [0,\pi/2]$.
Lind\'en \cite{Lin} proved the limiting case $\alpha=\pi $ of Theorem \ref{newlemma2} (1) with
  the constant $c_0\equiv \pi/\log(3)\,.$  For $c \in (1, c_0)\,,$ some of the level sets
$L(c) = \{z: k_G(z,1)/j_G(z,1)=c \}$ are displayed in Figure~1.

\item\label{rem-intro} Let $D\subset \Rn$ be a domain, and let $G=D\setminus\{x_0\}$ with $x_0\in D.$ For given $x,y\in G$ if there exists some constant $c\geq 1$ such that $\min\{d_D(x),d_D(y)\}\leq c \min\{|x-x_0|,|y-y_0|\}$, then by the definition of $j$-metric we have $j_G(x,y)\leq c j_D(x,y)$. We also see from \cite[Lemma 2.53]{Vu2} that $k_G(x,y)\leq c_1(c) k_D(x,y)$ with $c_1(c)$ depending only on $c$.

\end{enumerate}
\end{remark}

\begin{figure}[!ht]
\centering
\includegraphics[width=6cm]{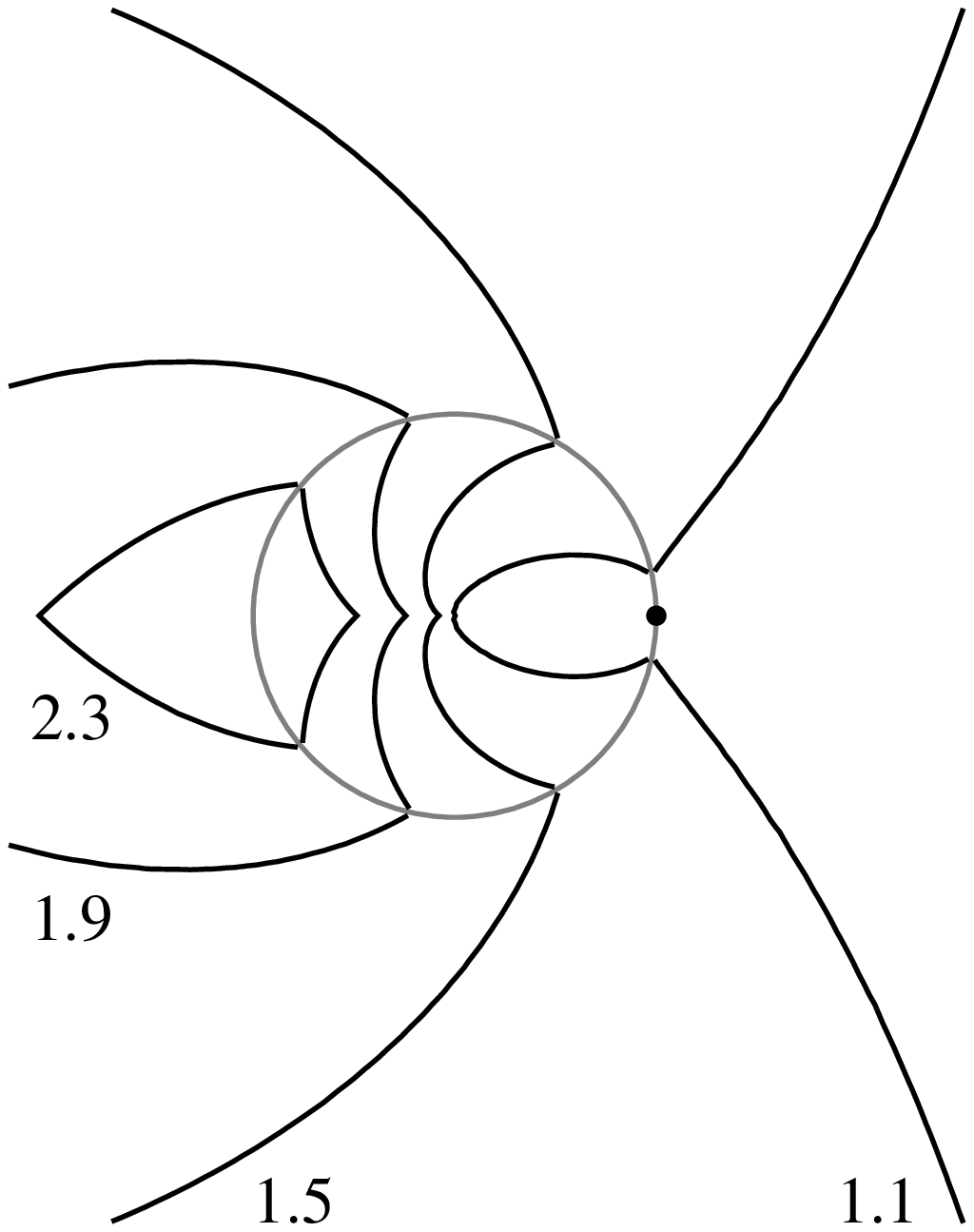}\hspace{5mm}
\includegraphics[width=6cm]{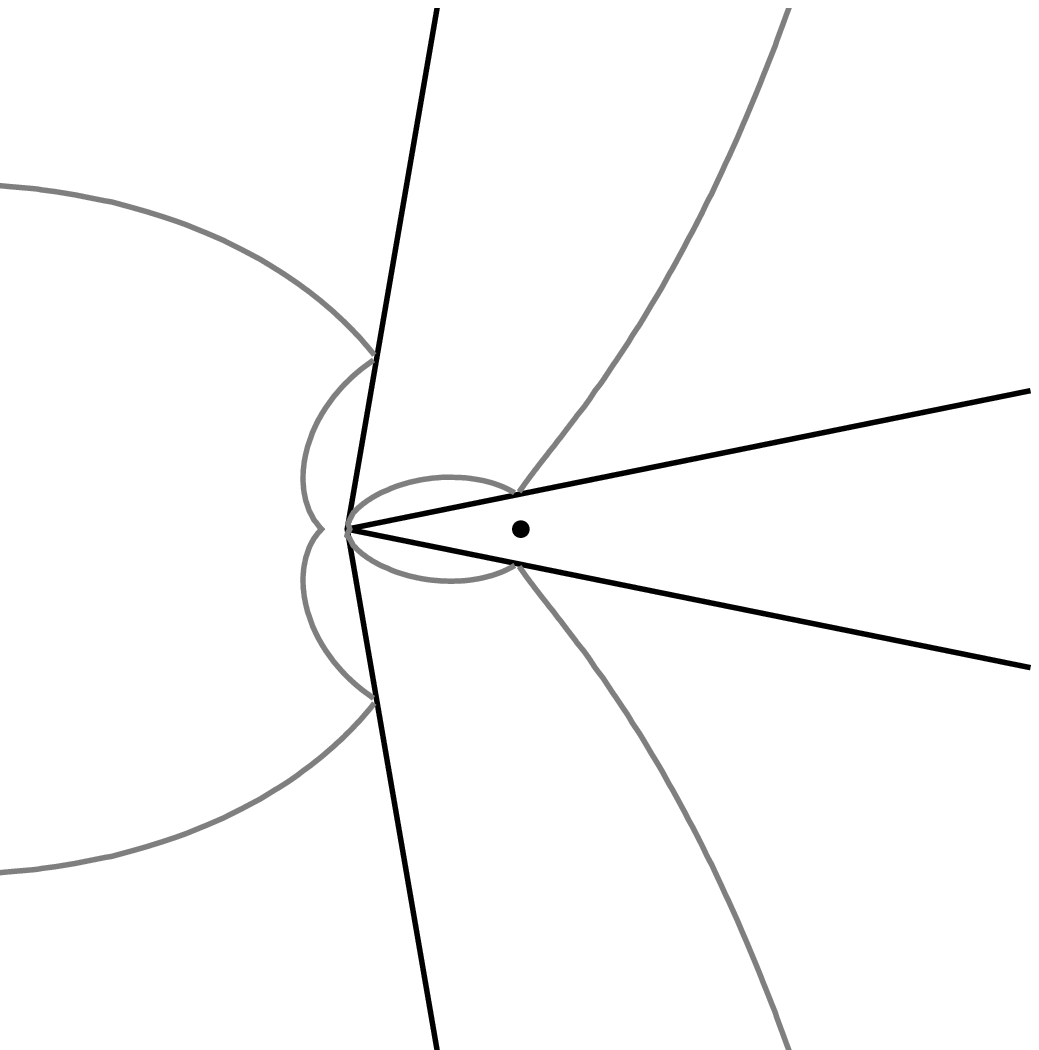}
\caption{Left: Level sets $L(c) = \{z: k_G(z,1)/j_G(z,1)=c \}$ for $G = \Rn \setminus \{ 0 \}$ and
$c = 1.1,\,1.5,\,1.9,\,2.3$. Right: Level sets $L(c)$ and angular domains as in Theorem \ref{newlemma2} (1) for $c=0.2,\,1.4$.}
\end{figure}

\section{Results concerning Problem \ref{prob2}}\label{sec-3}


In this section, our main goal is to study Problem \ref{prob2}, that is, to compare the Euclidean metric and the quasihyperbolic
metrics defined in a domain. Our main result is Theorem \ref{Jung-appl}.

In the next lemma, we recall a sharp inequality for the hyperbolic metric of the unit ball proved in
\cite[(2.27)]{Vu}.

\begin{lemma} \label{rhoineq}
For $x,y\in B^n$, let $t$ be as in $(\ref{rhodef})$. Then
$$ \tanh^2\frac{\rho_{B^n}(x,y)}2=\frac{|x-y|^2}{|x-y|^2+t^2}\, ,
$$
$$ |x-y|\leqslant 2\tanh\frac{\rho_{B^n}(x,y)}{4}=
\frac{2|x-y|}{\sqrt{|x-y|^2+t^2}+t}\, ,
$$
where equality holds for $x=-y$.
\end{lemma}

Earle and Harris \cite{EH} provided several applications of this
inequality and extended this inequality to other metrics such as the
Carath\'eodory metric. Notice that Lemma \ref{rhoineq} gives a sharp bound for
the modulus of continuity
\begin{equation*} 
id:\, (B^n,\rho_{B^n})\to (B^n,|\cdot|\,) \,.
\end{equation*}
For a $K$-quasiconformal homeomorphism
$$f: (B^n, \rho_{B^n}) \to (B^n, \rho_{B^n})$$
an upper bound for the modulus of continuity is well-known, see
\cite[Theorem 11.2]{Vu}. For $n=2$ the result is sharp for each
$K\ge 1$, see \cite[p. 65 (3.6)]{lv}. The particular case $K=1$
yields a classical Schwarz lemma.


As a preliminary step we record Jung's Theorem
\cite[Theorem~11.5.8]{Ber} which gives a sharp bound for the radius
of a Euclidean ball containing a given bounded domain.

\begin{lemma}\label{Jung-thm}
Let $G\subset \Rn$ be a domain with ${\rm diam}\,G < \infty$.
Then there exists $z\in \Rn$
such that $G\subset B^n(z,r)$, where $r\le \sqrt{n/(2n+2)}\,{\diam}\, G$.
\end{lemma}

\begin{theorem}\label{Jung-appl}
\begin{enumerate}
\item\label{ksv-sub22} If $x,y\in B^n$ are arbitrary and $w=|x-y|\,e_1/2$, then
$$k_{B^n}(x,y)\ge k_{B^n}(-w,w)=2\,k_{B^n}(0,w)=
2\,\log \frac{2}{2-|x-y|}\ge |x-y|\,,
$$
where the first inequality becomes equality when $y=-x$.
Moreover, the identity map $id:\, (B^n,k_{B^n})\to (B^n,|.|)$ has the sharp modulus of
continuity $\omega(t)=2(1-e^{-t/2})$.
\item\label{ksv-sub23} Let $G\psubset \Rn$ be a domain with
${\rm diam}\,G < \infty$ and $r=\sqrt{n/(2n+2)}\,{\diam}\, G$.
Then we have 
$$k_G(x,y)\ge 2\,\log \frac{2}{2-t}\ge t=|x-y|/r\,,
$$
for all distinct $x,y\in G$ with equality in the first step
when $G=B^n(z,r)$ and $z=(x+y)/2$.
Moreover, the identity map $id:\, (G,k_{G})\to (G,|.|)$ has the sharp modulus of
continuity $\omega(t)=2r(1-e^{-t/2})$.
\end{enumerate}
\end{theorem}
\begin{proof}
(1)
Without loss of generality, we assume that
$|x|\ge |y|$. 
We divide the proof into two cases.\\
{\it Case I:} The points $x$ and $y$ are both on a diameter of $B^n$.\\
If $0\in [x,y]$, by Proposition~\ref{k-equals-j}(\ref{ksv-sub1})
we have
$$k_{B^n}(x,y)=k_{B^n}(x,0)+k_{B^n}(0,y)=\log\frac{1}{(1-|x|)(1-|y|)}\,,
$$
and hence
$$k_{B^n}(-w,w) 
=2\log\frac{1}{1-|w|}\,.
$$
It is easy to verify that $k_{B^n}(x,y)\ge k_{B^n}(-w,w)$
is equivalent to $(|x|-|y|)^2\ge 0$\,.

If $y\in [x,0]$, then the proof goes in a similar way. Indeed,
in this situation
$$k_{B^n}(x,y)=\log\frac{1-|y|}{1-|x|}\ge k_{B^n}(-w,w)
$$
is equivalent to
$$(|x|-|y|)\left(1-\frac{1}{1-|y|}\right)\le \left(\frac{|x|-|y|}{2}\right)^2,
$$
which is trivial as the left hand term is $\le 0$. Equality clearly holds if $y=-x$.\\
{\it Case II:} The points $x$ and $y$ are arbitrary in $B^n$.\\
Choose $y'\in B^n$ such that $|x-y|=|x-y'|=2|w|$ with $x$ and $y'$ on a diameter of $B^n$. Then
$$k_{B^n}(x,y)\ge k_{B^n}(x,y')\ge k_{B^n}(-w,w)\,,
$$
where the first inequality holds trivially and the second holds by {\em Case I}. 
The sharp modulus of continuity can be obtained by a trivial rearrangement of the first inequality
from the statement.

(2)
Since $G$ is a bounded domain, by Lemma~\ref{Jung-thm},
there exists $z\in \Rn$ such that $G\subset B^n(z,r)$.
Denote $B:=B^n(z,r)$\,. Then the domain monotonicity property gives
$$k_G(x,y)\ge k_B(x,y)\,.
$$
Without loss of generality we may now assume that $z=0$.
Choose $u,v\in B$ in such a way that $u=-v$ and $|u-v|=2|u|=|x-y|$.
Hence by (\ref{ksv-sub22}) we have
$$k_G(x,y)\ge k_B(x,y)\ge k_B(-u,u)=2\,\log\frac{r}{r-|u|}\,.
$$
This completes the proof.
\end{proof}
A counterpart of Theorem~\ref{Jung-appl}
for the distance ratio metric $j_G$ can be formulated in the following form
(we omit the proofs, since they are very similar to the proofs of Theorem~\ref{Jung-appl}).

\begin{theorem}\label{Jung-appl-j}
\begin{enumerate}
\item If $x,y\in B^n$ are arbitrary and $w=|x-y|\,e_1/2$, then
$$j_{B^n}(x,y)\ge j_{B^n}(-w,w)= \log \frac{2+t}{2-t}= 2 {\rm artanh}(t/2)\ge t=|x-y|\,,
$$
where the first inequality becomes equality when $y=-x$.
Moreover, the identity map $id:\, (B^n,j_{B^n})\to (B^n,|.|)$ has the sharp modulus of
continuity $\omega(t)=2\tanh (t/2)$.
\item Let $G\psubset \Rn$ be a domain with ${\rm diam}\,G < \infty$ and
$r=\sqrt{n/(2n+2)}\,{\diam}\, G$.
Then we have
$$j_G(x,y)\ge \log \frac{2+t}{2-t} \ge t=|x-y|/r\,,
$$
for all distinct $x,y\in G$ with equality in the first step
when $G=B^n(z,r)$ and $z=(x+y)/2$.
Moreover, the identity map $id:\, (G,j_{G})\to (G,|.|)$ has the sharp modulus of
continuity $\omega(t)=2r\tanh (t/2)$.
\end{enumerate}
\end{theorem}
\section{Results concerning Problem \ref{pro}}\label{sec-4}


In this final section we present our results on
 Problem \ref{pro}.

\begin{theorem}\label{thm1}Let $G_1=\{(x,y)\in \Rt:|x|+|y|< 1\}$ and $G_2= \{ (x,y)\in \Rt:|x|^2+|y|^2< 1 \}$. Then \eqref{eq0} holds for $k_G$ metric with $c=\sqrt{2}$ but there is no constant $c>1$ for which \eqref{eq0} holds for the $j_G$ metric.\end{theorem}

\begin{proof}
Obviously,  $\partial G_1\cap\partial G_2 = \{ e_1,-e_1,e_2,-e_2 \}$ is a discrete set.
For each $w\in G_1$, we prove that \begin{equation}\label{eq-th-3}\delta_{G_2}(w)\geq \sqrt{2}\delta_{G_1}(w).\end{equation}
Without loss of generality, we may assume that $Re(w)\geq 0$ and $Im(w)\geq 0$. Then $Re(w)+Im(w)\leq 1$ and  $\delta_{G_1}(w)= \frac{1}{\sqrt{2}}(1-Re(w)-Im(w)).$ Hence, $$\delta_{G_2}(w)=1-\sqrt{Re(w)^2+Im(w)^2}\geq 1-Re(w)-Im(w)=\sqrt{2}\delta_{G_1}(w),$$ which proves equation \eqref{eq-th-3}.

Given $z_1,z_2 \in G_1$,
 let $\gamma$ be a quasihyperbolic geodesic joining $z_1$ and $z_2$ in $G_1$. Then by equation \eqref{eq-th-3},  $$k_{G_2}(z_1,z_2)\leq \int_{\gamma}\frac{|dw|}{\delta_{G_2}(w)}
 \leq \int_{\gamma}\frac{|dw|}{\sqrt{2}\delta_{G_1}(w)}=\frac{1}{\sqrt{2}}k_{G_1}(z_1,z_2).$$

 For the $j_G$ metric case, let $x_0=(1-\varepsilon,0)$, $y_0=(-1+\varepsilon,0)$ where $\varepsilon\in(0,1)$. Then $$|x_0-y_0|=2-2\varepsilon$$ and  $$\delta_{G_2}(x_0)=\delta_{G_2}(y_0)=\varepsilon=\sqrt{2}\delta_{G_1}(x_0)=\sqrt{2}\delta_{G_1}(y_0).$$ Hence,
 $$\frac{j_{G_1}(x_0,y_0)}{j_{G_2}(x_0,y_0)}=\frac{\log(1+\frac{\sqrt{2}(2-2\varepsilon)}{\varepsilon})}{\log(1+\frac{(2-2\varepsilon)}{\varepsilon})}\rightarrow 1, \;\;{\rm as} \;\; \varepsilon\rightarrow 0.$$
\end{proof}

\begin{theorem}\label{thm1'}
For $0<s<1$, let $G_1=\{(x,y): |x|^s+|y|^s< 1\}$ and $G_2=\{(x,y): |x|+|y|< 1\}$.  Then $k_{G_1}(z_1,z_2)\geq 2^{\frac{1}{s}-1}k_{G_2}(z_1,z_2)$ for all $z_1,z_2 \in G_1$.\end{theorem}

\begin{proof} We first prove that for all $w\in G_1$, $\delta_{G_2}(w)\geq 2^{\frac{1}{s}-1}\delta_{G_1}(w)$. Let $w=(a,b)\in G_1$. By symmetry, we only need to consider the case $0\leq b\leq a$. Denote $\gamma_s= \partial G_1\cap \{(x,y): x\geq 0,y\geq 0\}$, $\gamma_1=\partial G_2\cap \{(x,y): x\geq 0,y\geq 0\}$. Let $y_1\in \gamma_1$ be such that line $\ell_{0y_1}$, which goes through $0$ and $y_1$, is perpendicular to $\gamma_1$. Obviously, $\ell_{0y_1}\bot \gamma_s$, say at the point $y_2$. Let $y_3\in \gamma_1$ be such that $[w,y_3]\bot \gamma_1$, $y_4$ be the intersection point of $[w,y_3]$ and $\gamma_s$ and $w_1\in \ell_{0y_1}$ be such that $w_1$, $w$ and $e_1$ are collinear (see Figure \ref{fig1}).

\begin{figure}[!ht]
\centering
\includegraphics[width=0.4\textwidth,height=0.4\textwidth]{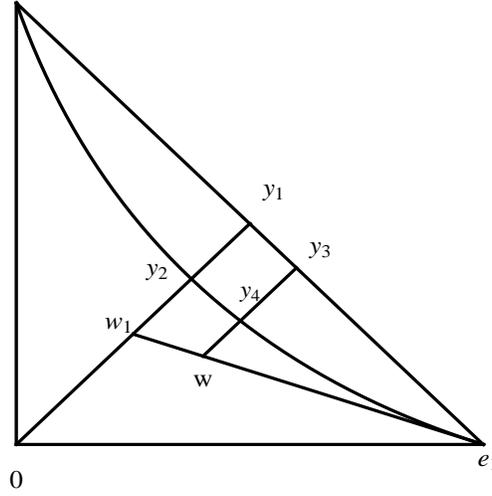}
\caption{Points $y_i$, $w$ and $w_1$ used in the proof of Theorem \ref{thm1'}. \label{fig1}}
\end{figure}
We observe first that \begin{equation}\label{eq2}\frac{\delta_{G_2}(w)}{\delta_{G_1}(w)}\geq \frac{|w-y_3|}{|w-y_4|}.\end{equation}

By similar triangle property, we can get $$\frac{|w-y_3|}{|w_1-y_1|}=\frac{|e_1-w|}{|e_1-w_1|}\geq \frac{|w-y_4|}{|w_1-y_2|},$$ which together with \eqref{eq2} and simple calculation show that $$\frac{\delta_{G_2}(w)}{\delta_{G_1}(w)}\geq \frac{\delta_{G_2}(w_1)}{\delta_{G_1}(w_1)}\geq \frac{\delta_{G_2}(0)}{\delta_{G_1}(0)}=2^{\frac{1}{s}-1}.$$

Given $z_1,z_2\in G_1$, let $\beta$ be a quasihyperbolic geodesic joining $z_1$ and $z_2$ in $G_1$. Then $$k_{G_2}(z_1,z_2)\leq \int_{\beta}\frac{|dw|}{\delta_{G_2}(w)}\leq 2^{\frac{1}{s}-1}k_{G_1}(z_1,z_2).$$

\end{proof}

We generalize the above two Theorems into the following conjecture.

\begin{conjecture}For $0<s<t$, let $G_1=G_s=\{(x,y): |x|^s+|y|^s\leq 1\}$, $G_2=G_t=\{(x,y): |x|^t+|y|^t\leq 1\}$. We conjecture that $k_{G_1}(z_1,z_2)\geq 2^{\frac{1}{s}-\frac{1}{t}}k_{G_2}(z_1,z_2)$ for all $z_1,z_2 \in G_1$.
\end{conjecture}

The following result gives a solution to Problem \ref{pro}.

\begin{theorem}\label{thm2}
Let $G_1$ be a bounded subdomain of the domain $G_2 \subsetneq \Rn$. Then for all $x,y\in G_1$
\[
  m_{G_1}(x,y)\geq c m_{G_2}(x,y),
\]
where $m_{G_i} \in \{ j_{G_i},k_{G_i} \}$ for $i=1,2$ and $c=1+\frac{2 \dist(G_1,\partial G_2)}{\diam(G_1)}$.
\end{theorem}

\begin{proof}We first prove the $j_G$ metric case.

For each $x,y\in G_1$, we are going to prove
$$
\log \left( 1+\frac{|x-y|}{ \min \{ \delta_{G_1}(x) , \delta_{G_1}(y) \} } \right) \geq \left( 1+\frac{2 \dist(G_1,\partial G_2)}{\diam(G_1)} \right) \log \left( 1+\frac{|x-y|}{ \min \{ \delta_{G_2}(x) , \delta_{G_2}(y) \} }\right)
$$

Since $\delta_{G_2}(w)\geq \delta_{G_1}(w)+\dist(G_1,\partial G_2)$ holds for all $w\in G_1$, then it suffices to prove

 \begin{eqnarray*}\diam(G_1)\log \left( 1+\frac{|x-y|}{\min \{ \delta_{G_1}(x) , \delta_{G_1}(y) \} } \right) \ge  (\diam(G_1)+2 \dist(G_1,\partial G_2))\\
\cdot \log \left( 1+\frac{|x-y|}{ \min \{ \delta_{G_1}(x)+\dist(G_1,\partial G_2) , \delta_{G_1}(y)+\dist(G_1,\partial G_2) \} } \right).
 \end{eqnarray*}

Let $$f(z)= (d+2z)\log(1+\frac{a}{b+z}),$$ where $d=\diam(G_1)$, $a=|x-y|$, $b=\delta_{G_1}(x)$ and $z\geq 0$.

Then
$$f'(z)=\log(1+\frac{a}{z+b})-\frac{a(2z+d)}{(z+a+b)(z+b)},$$ and $$f''(z)=\frac{-2a(z+a+b)(z+b)+a(2z+d)(2z+2b+a)}{(z+a+b)^2(z+b)^2}.$$

Denoting $$h(z)=-2a(z+a+b)(z+b)+a(2z+d)(2z+2b+a),$$  it is easy to see that $$h'(z)=a(2d-2b-a+2z)>0$$ which implies that $h(z)>h(0)>0$.

Hence $f''(z)\geq 0$ which yields  $$f'(z)\leq f'(\infty)=0$$ and hence the function $f(z)$ is decreasing. Thus the assertion follows.

For the $k_G$ metric case, we first prove that for each $w\in D_1$, the following inequality holds: $$\delta_{G_2}(w)\geq (1+\frac{2\dist(G_1,\partial G_2)}{\diam(G_1)})\delta_{G_1}(w).$$

In fact, for each $w\in G_1$ we have
 \begin{eqnarray*}\diam(G_1)\delta_{G_2}(w)&\!\! \ge \!\!&  \diam(G_1)(\delta_{G_1}(w)+\dist(G_1,\partial G_2))\\&\!\! \ge \!\!& (\diam(G_1)+2\dist(G_1,\partial G_2))\delta_{G_1}(w).
 \end{eqnarray*}

Given $x,y \in G_1$, let $\gamma$ be a quasihyperbolic geodesic joining $x$ and $y$ in $G_1$. Then $$k_{G_2}(x,y)\leq \int_{\gamma}\frac{|dw|}{\delta_{G_2}(w)}\leq \int_{\gamma}\frac{|dw|}{c\delta_{G_1}(w)}\leq \frac{1}{c}k_{G_1}(x,y),$$
where $c=1+\frac{2\dist(G_1,\partial G_2)}{\diam(G_1)}$.

\end{proof}

\begin{corollary}Let $0<r<R$ and $G_1=B^n(x,r)$, $G_2=B^n(x,R)$. Then for all $x,y\in G_1$
\[
  m_{G_1}(x,y)\geq c m_{G_2}(x,y),
\]
where $m_{G_i} \in \{ j_{G_i},k_{G_i} \}$ for $i=1,2$ and for $c=R/r$.
\end{corollary}

%
%
%
%
%
%
%

{\bf Acknowledgement.} This research was finished when the first author was an academic visitor
in Turku University and the first author was supported by the Academy of Finland grant of Matti Vuorinen with the
Project number 2600066611. She thanks
Department of Mathematics in Turku University for hospitality.


\begin{thebibliography}{HIMPS}

\bibitem{AVV}  {\sc G.D. Anderson, M.K. Vamanamurthy, and M.K.
Vuorinen,}
{\em Conformal Invariants, Inequalities, and Quasiconformal Maps},
John Wiley \& Sons, Inc., 1997.

\bibitem{Ber}  {\sc M. Berger,} {\em Geometry I},
Springer-Verlag, Berlin, 1987.

\bibitem{EH}  {\sc C.J. Earle and L.A. Harris,}
Inequalities for the Carath\'eodory and Poincar\'e metrics in open unit balls,
{\em Pure Appl. Math. Q.}~\textbf{ 7} (2011), no. 2, Special Issue: In honor of Frederick W. Gehring, Part 2, 253--273.

\bibitem{Geh99}  {\sc F.W. Gehring,}
Characterizations of quasidisks,
{\em Quasiconformal geometry and dynamics}, \textbf{48} (1999), 11--41.

\bibitem{GO}  {\sc F.W. Gehring and B.G. Osgood,}
Uniform domains and the quasihyperbolic metric,
{\em J. Anal. Math.}~\textbf{36} (1979), 50--74.

\bibitem{GP}  {\sc F.W. Gehring and B.P. Palka,}
Quasiconformally homogeneous domains,
{\em J. Anal. Math.}~\textbf{30} (1976), 172--199.

\bibitem{HIMPS} {\sc P. H\"ast\"o, Z. Ibragimov, D. Minda, S. Ponnusamy and S. Sahoo,}
Isometries of some hyperbolic-type path metrics, and the hyperbolic medial axis, In the tradition of Ahlfors-Bers. IV,  63--74, {\em Contemp. Math.}~\textbf{432}, Amer. Math. Soc., Providence, RI, 2007.

\bibitem{HPS}  {\sc P. H\"ast\"o, S. Ponnusamy and S.K. Sahoo,}
Inequalities and geometry of the Apollonian and related metrics,
{\em Rev. Roumaine Math. Pures Appl.} {\bf 51} (2006), 433--452.

\bibitem{KlenB}  {\sc R. Kl\'en,}
Local Convexity Properties of Quasihyperbolic Balls in Punctured Space, {\em  J. Math. Anal. Appl.} { \bf 342}, 2008, 192--201.

\bibitem{Klen}  {\sc R. Kl\'en,}
On hyperbolic type metrics, Dissertation, University of Turku, Helsinki, 2009 {\em Ann. Acad. Sci. Fenn. Math. Diss.}
{\bf 152} (2009), 49pp.

\bibitem{lv}  {\sc O. Lehto and K. I. Virtanen,} Quasiconformal mappings in the plane.
Second edition. Translated from the German by K. W. Lucas. Die
Grundlehren der mathematischen Wissenschaften, Band 126.
Springer-Verlag, New York-Heidelberg, 1973. viii+258 pp.

\bibitem{Lin}  {\sc H. Lind\'en,}
Quasihyperbolic geodesics and uniformity in elementary domains,
Dissertation, University of Helsinki, Helsinki, 2005.
{\em Ann. Acad. Sci. Fenn. Math. Diss.}  No.~\textbf{146} (2005), 50 pp.

\bibitem{Mac96}  {\sc P. MacManus,}
The complement of a quasim\"obius sphere is uniform,
{\em Ann. Acad. Sci. Fenn. Math.}~\textbf{21} (1996), 399--410.

\bibitem{mo} {\sc G.J. Martin and B.G. Osgood,}
The quasihyperbolic metric and the associated estimates on the
hyperbolic metric, \emph{J. Anal. Math.} {\bf 47} (1986), 37--53.

\bibitem{MS79}  {\sc O. Martio and J. Sarvas,}
Injectivity theorems in plane and space, {\em Ann. Acad. Sci. Fenn.
Math.}~\textbf{4} (1979), 384--401.

\bibitem{RasilaTalponen12} {\sc A. Rasila and J. Talponen,}
Convexity properties of quasihyperbolic balls on Banach spaces, {\em Ann. Acad. Sci. Fenn. Math.}~\textbf{37} (2012), 215--228.

\bibitem{Va88}  {\sc J. V\"ais\"al\"a,}
Uniform domains,
{\em Tohoku Math. J.} \textbf{40} (1988),  101--118.

\bibitem{Va91} {\sc J. V\"ais\"al\"a,}
Free quasiconformality in Banach spaces II,
{\em Ann. Acad. Sci. Fenn. Math.} \textbf{16} (1991), 255--310.

\bibitem{Va98} {\sc J. V\"ais\"al\"a,}
Relatively and inner uniform domains,
{\em Conformal Geometry and Dynamics,} \textbf{2} (1998), 56--88.

\bibitem{Va05}  {\sc J. V\"ais\"al\"a,}
Quasihyperbolic geodesics in convex domains,
{\em Result. Math.}~\textbf{48} (2005), 184--195.

\bibitem{Vu1}  {\sc M. Vuorinen,}
Capacity densities and angular limits of quasiregular mappings. 
{\em Trans. Amer. Math. Soc.} 263 (1981), no. 2, 343--354.

\bibitem{Vu2}  {\sc M. Vuorinen,}
Conformal invariants and quasiregular mappings,
{\em J. Anal. Math.}~\textbf{45} (1985), 69--115.

\bibitem{Vu}  {\sc M. Vuorinen,}
{\em Conformal Geometry and Quasiregular Mappings},
Lecture Notes in Mathematics~1319,
Springer-Verlag, Berlin--Heidelberg--New York, 1988.

\bibitem{Vu07}  {\sc M. Vuorinen,}
Metrics and quasiregular mappings, In {\em Quasiconformal Mappings
and their Applications} (New Delhi, India, 2007), S. Ponnusamy, T.
Sugawa, and M. Vuorinen, Eds., Narosa Publishing House, pp.
291--325.
\end{thebibliography}
\end{document}